\documentclass[10pt]{amsart}
\usepackage{amssymb,amsmath}
\usepackage{mathtools}
\usepackage{graphicx}

\begin{document}

\newtheorem{thm}{Theorem}[section]
\newtheorem{lemma}[thm]{Lemma}
\newtheorem{defin}[thm]{Definition}
\newtheorem{rmk}[thm]{Remark}
\newtheorem{conj}[thm]{Conjecture}
\newtheorem{cor}[thm]{Corollary}
\newtheorem{mr}[thm]{Main Result}
\newtheorem{ass}[thm]{Assumption}
\newtheorem{prop}[thm]{Proposition}
\newtheorem{qu}[thm]{Question}

\makeatletter

\newcommand{\explain}[2]{\underset{\mathclap{\overset{\uparrow}{#2}}}{#1}}
\newcommand{\explainup}[2]{\overset{\mathclap{\underset{\downarrow}{#2}}}{#1}}

\makeatother

\newcommand{\map}{\mbox{$\rightarrow$}}
\newcommand{\bbb}{\mbox{$\beta$}}
\newcommand{\la}{\mbox{$\lambda$}}
\newcommand{\aaa}{\mbox{$\alpha$}}
\newcommand{\eee}{\mbox{$\epsilon$}}
\newcommand{\Rrr}{\mbox{$\mathbb{R}$}}
\newcommand{\lpd}{\mbox{$L^{\mathcal{V}(P,D^*)}$}}
\newcommand{\fpd}{\mbox{$\mathcal{V}(P,D^*)$}}
\newcommand{\bdd}{\mbox{$\partial$}}

\newcommand{\Li}{\mbox{$L_+^{in}$}}
\newcommand{\Lo}{\mbox{$L_+^{out}$}}

\title{Companions of the unknot and width additivity}
\author{Ryan Blair}
\author{Maggy Tomova}
\thanks{Research partially supported by an NSF grant.}
\begin{abstract}

It has been conjectured that for knots $K$ and $K'$ in $S^3$, $w(K\#K')= w(K)+w(K')-2$. In \cite{ST}, Scharlemann and Thompson proposed potential counterexamples to this conjecture. For every $n$, they proposed a family of knots $\{K^n_i\}$ for which they conjectured that $w(B^n\#K^n_i)=w(K^n_i)$ where $B^n$ is a bridge number $n$ knot. We show that for $n>2$ none of the knots in $\{K^n_i\}$ produces such counterexamples. \end{abstract}
\maketitle

\section{Introduction and definitions}

The width of a knot is an invariant first defined by Gabai \cite{Ga} in his proof of property $R$. The width of a projection of a knot in $S^3$ is an even number which depends on the number of critical points as well as on their relative heights. The width of the knot is the minimum width over all projections.

In this paper we often need to make a distinction between the width of a knot and the width of a particular diagram of the knot. We will use standard font to denote a projection of a knot and script font to refer to the family of all projections of the knot. In particular, $w(\mathcal{K})\leq w(K)$ for any $K$ a projection of $\mathcal{K}$. If a projection of the knot achieves the width of the knot, we say that the knot is in thin position. Amongst the applications of thin position have been Gordon and Luecke's proof of the knot complement conjecture \cite{GL}, Rieck and Sedgwick's study of the behavior of Heegaard surfaces under Dehn surgery \cite{RS2}, and Scharlemann and Thompson's proof of Waldhausen's Theorem \cite{ST2}. However, computing the width of a particular knot remains almost always impossible.

One of the questions regarding width that has attracted much interest is its behavior under connect sum. It is not difficult to see that $$w(\mathcal{K}\#\mathcal{K}')\leq w(\mathcal{K})+w(\mathcal{K}')-2.$$ Whether $w(\mathcal{K}\#\mathcal{K}')= w(\mathcal{K})+w(\mathcal{K}')-2$, however, remains an open question. There are partial results and special cases that point to equality. Most notably, Scharlemann and Schultens \cite{SchSch} showed that $w(\mathcal{K}\#\mathcal{K}')\geq max\{w(\mathcal{K}),w(\mathcal{K}')\}$ and Rieck and Sedgwick \cite{RS} showed that the equality $w(\mathcal{K}\#\mathcal{K}')= w(\mathcal{K})+w(\mathcal{K}')-2$ holds for small knots. Because of its similarity to bridge number, the just-mentioned partial results and the failed search for a counterexample, it was believed that for any two knots $\mathcal{K}$ and $\mathcal{K}'$, $w(\mathcal{K}\#\mathcal{K}')= w(\mathcal{K})+w(\mathcal{K}')-2$.

\begin{figure}
\begin{center} \includegraphics[scale=.4]{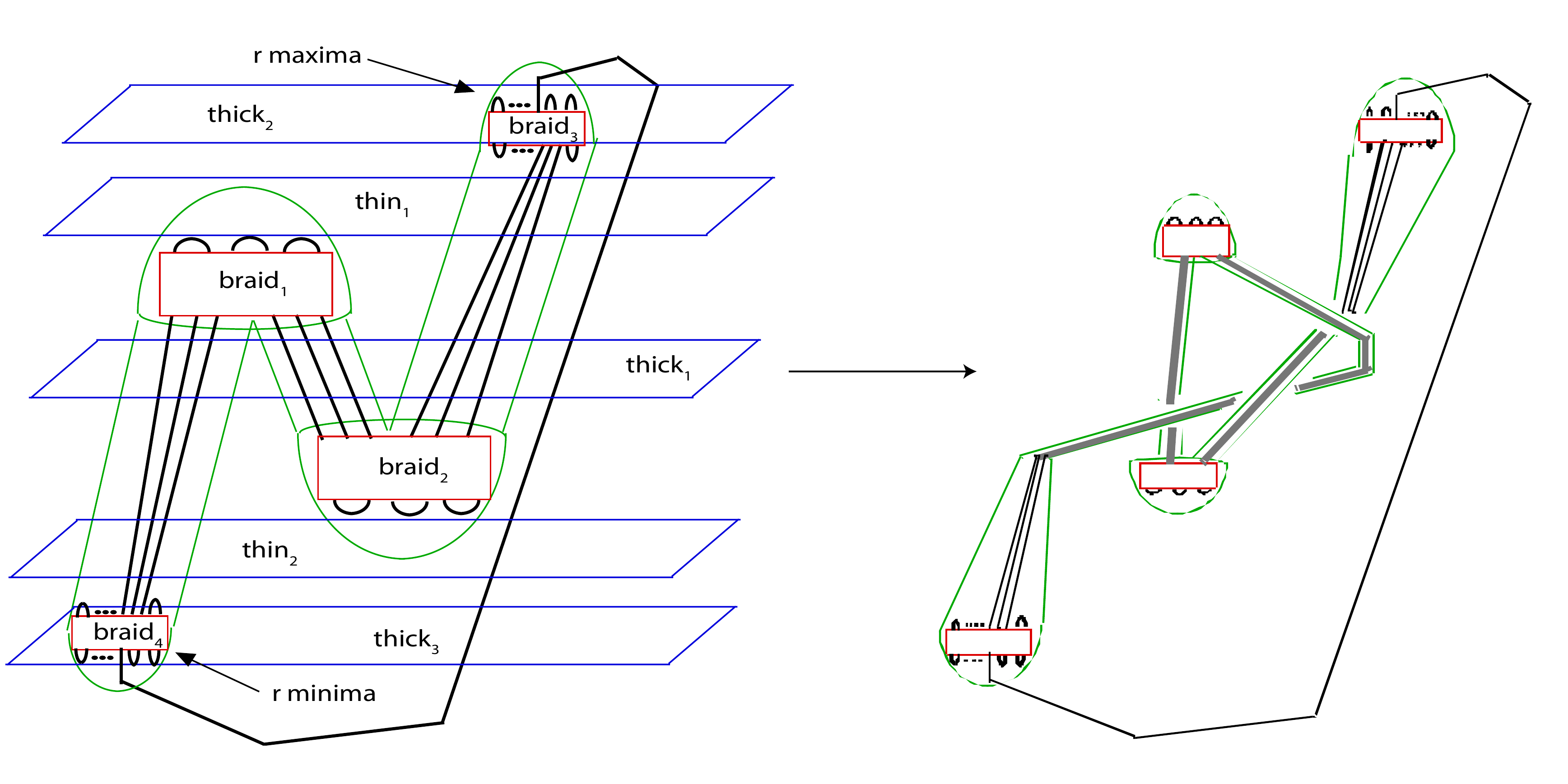}
\end{center}
\caption{} \label{fig:counter} \end{figure}

In a surprising paper \cite{ST}, Scharlemann and Thompson proposed examples for which they conjectured that the equality $w(\mathcal{K}\#\mathcal{K}')= max\{w(\mathcal{\mathcal{K}}),w(\mathcal{K}')\}$ holds. For each $n$, they gave an infinite family of knots that have a projection of the form given in Figure \ref{fig:Kn}.

\begin{conj} \cite{ST}
For each $n$, there exists at least one knot $\mathcal{K}^n$ with a projection $K^n$ as in Figure \ref{fig:Kn} so that $w(K^n)=w(\mathcal{K}^n)$.

\end{conj}

\begin{cor}
For each knot $\mathcal{K}^n$ as in the conjecture, $w(\mathcal{B}^n\#\mathcal{K}^n)=w(\mathcal{K}^n)$ where $\mathcal{B}^n$ is a bridge number $n$ knot for which thin and bridge position coincide.
\end{cor}

If we assume the conjecture, Figure \ref{fig:counter} gives a proof of the corollary where $n=2$ and $\mathcal{B}^2$ is the trefoil. For more details see \cite{ST}. Note that even if $K^n$ is not in thin position but satisfies the weaker condition that  $w(K^n)<w(\mathcal{K}^n)+w(\mathcal{B}^n)-2$, it will still give a counterexample to $w(\mathcal{K}\#\mathcal{K}')= w(\mathcal{K})+w(\mathcal{K}')-2$, although the corollary will no longer follow.

In this paper, we show that for $n>2$ all knots in the families used in \cite{ST} and depicted in Figure \ref{fig:Kn} satisfy $w(K^n)\geq w(\mathcal{K}^n)+w(\mathcal{B}^n)-2$. It follows that the projection $K^n\#B^n$ for which $w(K^n\#B^n)=w(K^n)$ discovered by Scharlemann and Thompson satisfies the inequality  $w(K^n\#B^n)\geq w(\mathcal{K}^n)+w(\mathcal{B}^n)-2$ and cannot be used to provide an interesting upper bound for $w(\mathcal{K}^n\#\mathcal{B}^n)$. Invalidating these likely counterexamples is additional evidence supporting the conjecture that in fact width is additive under connect sums.

In the final section, we investigate a larger family of knots that includes those in \cite{ST}. In particular, we look at projections of wrapping number one companions of the unknot $U$, where $U$ is in bridge position with $n$ bridges. See Figure \ref{fig:wrappingtorus}. Any such projection in thin position with $n>1$ would provide a counterexample to the additivity of width. We show that the techniques developed in the paper can be easily adapted to prove that many of these more general projections can be thinned.

\section{Preliminaries}

Recall that we will use standard font to denote a projection of a knot and script font to refer to the family of all projections of the knot. Suppose $K$ is a projection of $\mathcal{K}$ that is in general position with respect to $\pi$, the standard height function on $S^3$. If $t$ is a regular
value of $\pi|_K$, $\pi^{-1}(t)$ is called a level sphere with width $w(\pi^{-1}(t))=|K\cap \pi^{-1}(t)|$. If $c_{0}<c_{1}<...<c_{n}$ are all the critical values of $\pi|_K$, choose regular values
$r_{1},r_{2},...,r_{n}$ such that $c_{i-1}<r_{i}<c_{i}$. Then the {\em width of $K$ } is defined by $w(K)=\sum w(\pi^{-1}(r_{i}))$. The {\em width} of $\mathcal{K}$, $w(\mathcal{K})$, is the minimum of $w(K)$ over all $K\in \mathcal{K}$. We say that $K$ is in thin position if $w(\mathcal{K})=w(K)$.  More details about thin position and basic results can be found in \cite{Schar1}.

A level sphere that corresponds to a local minimum in the ordered sequence of integers $\{ w(\pi^{-1}(r_{1}), w(\pi^{-1}(r_{2})),..., w(\pi^{-1}(r_{n}))\}$ is called {\em thin}
and a level sphere that corresponds to a local maximum is called {\em thick}. We will use the following result found in \cite{SchSch} to simplify our computations.

\begin{lemma}\label{lem:thickandthin}
If $\{a_i\}$, $i=0,...n$ and $\{b_j\}$, $j=0,...n+1$ are the widths of all thin and thick spheres for $K$ respectively, then $w(K)=(\Sigma b_j^2-\Sigma a_i^2)/2$.
\end{lemma}

It is often useful to decompose a knot into tangles. We will need the following special case of this approach. Suppose $P$ and $P'$ are two non-parallel level spheres for a knot $K$ with the same width and let $T \subset K$ be a tangle lying between them such that, between $P$ and $P'$, $K / T$ consists only of vertical strands, see Figure \ref{fig:KT}. We can extend the definition of width to also apply to a tangle $T \subset D^2 \times I$ in the following way. If $c_{0}<c_{1}<...<c_{n}$ are all the
critical values for $T$, choose regular values
$r_{1},r_{2},...,r_{n}$ such that $c_{i-1}<r_{i}<c_{i}$ for $i=1,...n$. Then $w(T)=\sum
w(\pi^{-1}(r_{i}))$. (Note that this definition does not add the widths of $D^2 \times \{0\}$ and  $D^2 \times \{1\}$ to the total number. This has the somewhat unpleasant consequence that the width of a tangle with a single critical point is 0. However, this definition is the most convenient for our purposes.) We can also associate a knot $K^-$ to $K/T$ by identifying $P$ and $P'$. This operation is not well defined, but the width of the resulting projection is. In this case, we can express the width of $K$ in terms of the widths of $T$ and $K/T$ where we define $w(K/ T)=w(K^-)$.

\begin{figure}
\begin{center} \includegraphics[scale=.4]{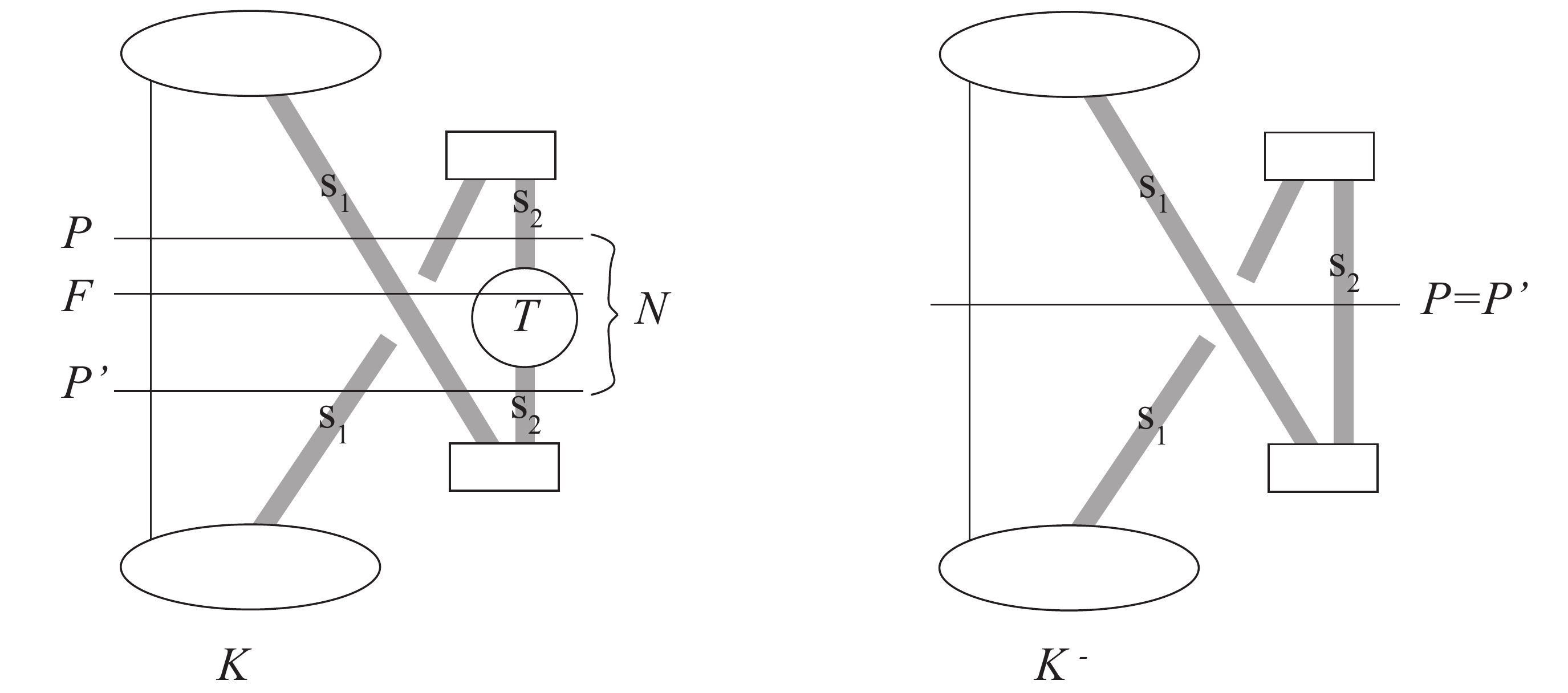}
\end{center}
\caption{} \label{fig:KT} \end{figure}

\begin{lemma}\label{lem:addingtangle}
$w(K)=w(K/T)+w(T)+l(r-1)+w(P)$ where $r$ is the number of critical points for $T$ and $l=|P \cap K| -|T\cap (D^2 \times \{0\})|$. \end{lemma}

\begin{proof}

To obtain $K$ from $K^-$ we can imagine inserting a copy $N$ of $S^2 \times I$ containing $T$ and $l$ vertical arcs just below $P$. $N$ contains all additional level spheres we need take into account to obtain $w(K)$ from $w(K^-)$. Let $F$ be a level sphere contained in $N$, see Figure \ref{fig:KT}. It has some punctures coming from $T$ and $l$ punctures coming from $K/T$. There are $r-1$ such spheres and the sum of their widths is then $w(T)+l(r-1)$. Also in $K$, there are two spheres, $P$ and $P'$, with equal width and only one of them is accounted for in $w(K^-)$. Thus, $w(K)=w(K/T)+w(T)+l(r-1)+w(P)$ as desired.

\end{proof}

The following lemma is often used but we give a proof here for completeness.

\begin{lemma}\label{lem:out}

Suppose a tangle $T\subset (D^2 \times I)$ is in bridge position (all maxima are above all minima) and $T \cap (D^2 \times \{1\})=\emptyset$. Then there is an isotopy $\phi$, such that $\phi(T)$ is in bridge position, $
w(\phi(T))=w(T)$ and the diagram for $\phi(T)$ includes a sub-strand of the tangle with one endpoint in $T \cap (D^2 \times \{0\})$ containing a single maximum and no intersections with itself or any other part of $T$. \end{lemma}

\begin{figure}
\begin{center} \includegraphics[scale=.5]{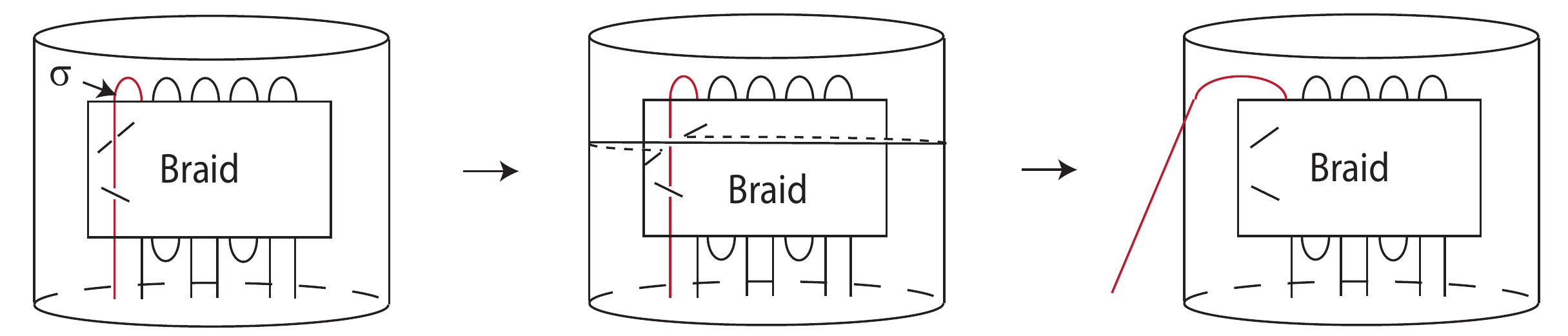}
\end{center}
\caption{} \label{fig:getout} \end{figure}

\begin{proof}
Let $\sigma$ be any sub-strand of $T$ with one endpoint of $\sigma \subset T \cap (D^2 \times \{0\})$ and the other endpoint lying just past the first critical point (necessarily a maximum), see Figure \ref{fig:getout}. Suppose there are some crossings where $\sigma$ is the overstrand. Consider the highest such crossing  and let $u$ be a small segment of $T$ containing the understrand at that crossing. Isotope a small neighborhood $\gamma$ of $u$ to lie in $\bdd (D^2 \times I)$. The arc $\gamma$ together with an arc $\gamma' \subset \bdd D^2 \times I$ cobound a circle with a single minimum and a single maximum coinciding with the endpoints of $\gamma$ which is the boundary of a subdisk $D'$ of $(\bdd D^2 \times I)\cup (D^2 \times \{1\})$. This disk gives an isotopy between $\gamma$ and $\gamma'$. The isotopy preserves width and decreases the number of overcrossings of $\sigma$, although, it may create numerous other crossings not involving $\sigma$. After finitely many iterations, we may assume that all crossings of $\sigma$ are undercrossings. In particular, we can isotope $\sigma$ to be disjoint from the rest of $T$.
\end{proof}

\section{Results}

In all figures in this paper, an oval represents any tangle and a rectangle represents a tangle in bridge position (i.e., a tangle for which all maxima are above all minima). We will call a tangle contained in $D^2 \times I$ that is in bridge position a {\em braid box}. We do not require that a braid box contain both minima and maxima.

\begin{defin} A knot is of {\em type $n$} if it has a projection as in Figure \ref{fig:Kn}. In particular, the following hold:

\begin{enumerate}
\item Each braid box $X_{i,1}$ is higher than and disjoint from braid box $X_{i+1,1}$.
\item Each braid box $X_{i,2}$ is lower than and disjoint from braid box $X_{i+1,2}$.

\item $X_{i,1}$ is higher than and disjoint from $X_{j,2}$ for every $i$ and $j$.

\item The number of strands descending out of $X_{i,1}$ to the right is equal to the number of strands ascending out of $X_{i,2}$ to the right.

\end{enumerate}
\end{defin}

It is clear that a knot can be of more than one type. In particular, if a knot is of type $n$, then it is also of type $m$ for all $m \leq n$. We will show that under certain conditions the projection given in Figure \ref{fig:Kn}  is not thin. To do this, we will describe two isotopies and for each one we will compare the width of the projection after the isotopy to the width of the original projection. The isotopies are depicted in Figures \ref{fig:Kn-1} and \ref{fig:Kn3}.

\begin{figure}
\begin{center} \includegraphics[scale=.3]{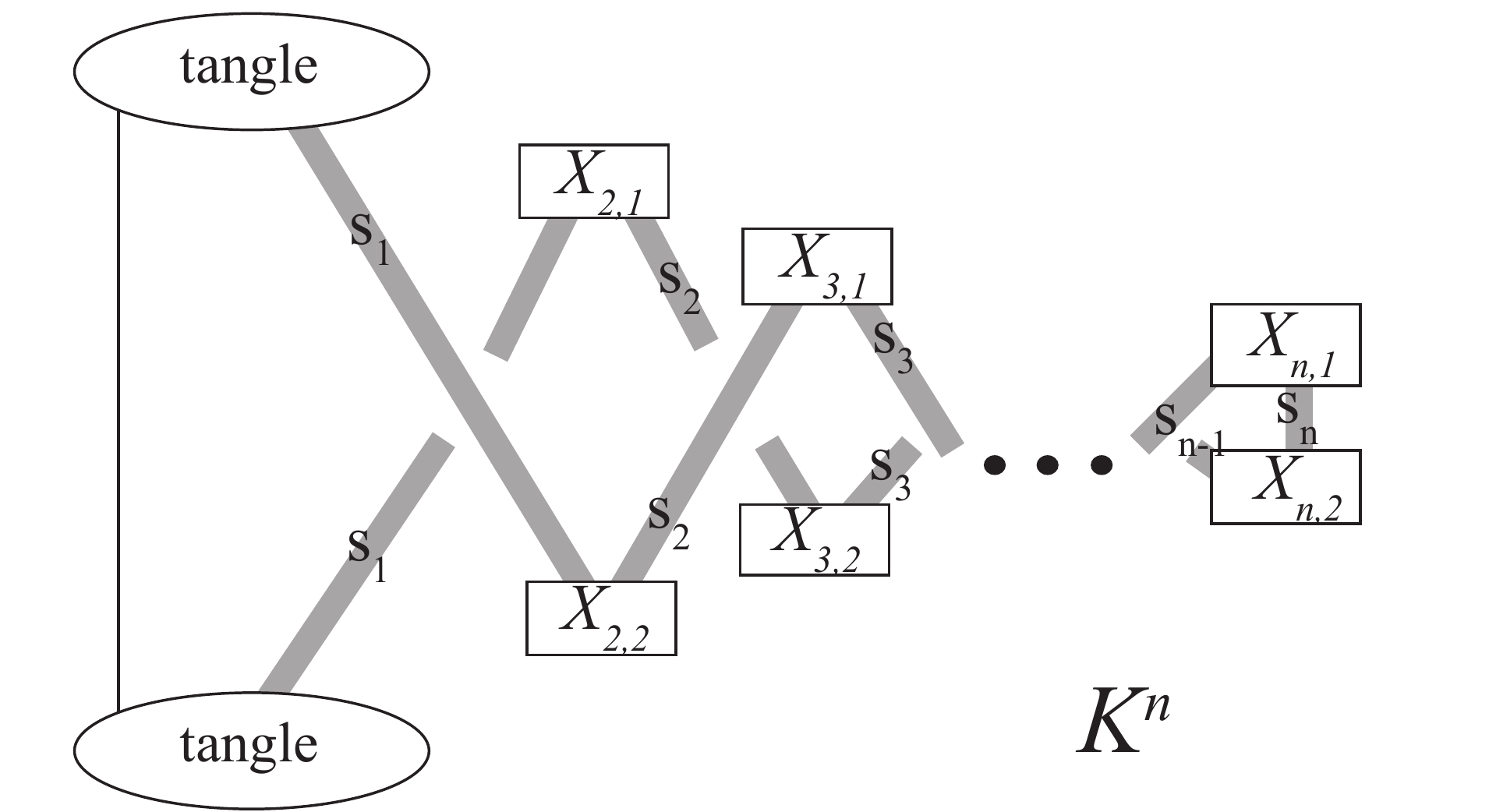}
\end{center}
\caption{} \label{fig:Kn} \end{figure}

Consider first the isotopy depicted in Figure \ref{fig:Kn-1} (note that the middle circle is a tangle containing braid boxes $\{X_{4,1},...,X_{n,2}\}$). For each $i$, $s_{i}$ (illustrated in grey) is the number of strands connecting braid box $X_{i,2}$ to braid box $X_{i+1,1}$ and braid box $X_{i,1}$ to braid box $X_{i+1,2}$. The letters $a$ and $b$ represent the maximum number of intersections between any level sphere and braid boxes $X_{2,1}$ and $X_{2,2}$ respectively. We will be using Lemma \ref{lem:thickandthin} to compute the widths of knots. Notice that the width is only changed in the regions affected by the isotopy. If each of $X_{2,1}$ and $X_{2,2}$ has both minima and maxima, then the thick spheres disjoint from $T$ and affected by the isotopy are exactly the two level spheres that intersect $X_{2,1}$ and $X_{2,2}$ in $a$ and $b$ points respectively and the thin spheres affected are the two level spheres directly above and below $T$. The first step of the isotopy relies on Lemma \ref{lem:out}. We will also use the following easy-to-verify inequalities.

\begin{rmk}\label{rmk:easy} Let $K^n$ be as in Figure \ref{fig:Kn-1}. Then
\begin{enumerate}
\item $a,b \geq s_1+s_2$,
\item if the number of critical points in $T$ is $r$, then $r\geq s_2+2s_3+2s_3+...+2s_{n-1}+s_n$.
\end{enumerate}
\end{rmk}

\begin{figure}
\begin{center} \includegraphics[scale=.3]{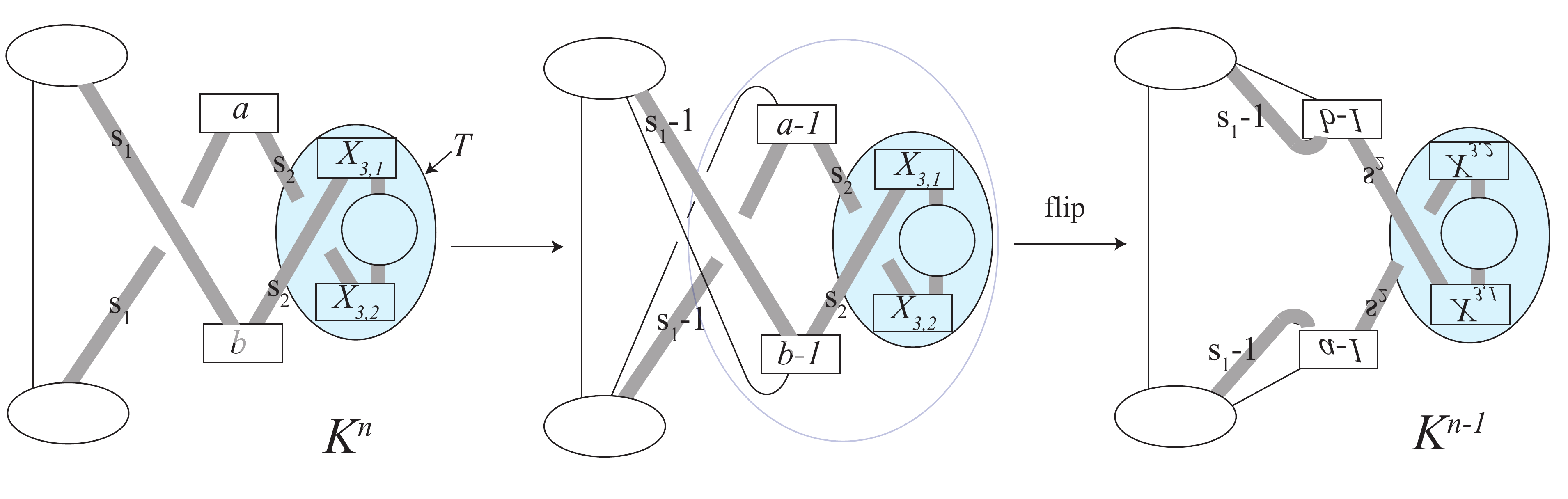}
\end{center}
\caption{} \label{fig:Kn-1} \end{figure}

\begin{lemma}\label{lem:s1leqs3}
Let $K^n$ and $K^{n-1}$ be the projections of knots of type $n$ and $n-1$ depicted in Figure \ref{fig:Kn-1}. Then $K^n$ and $K^{n-1}$ are projections of the same knot. Furthermore, if $n \geq 3$, $s_i \geq 3$ for all $i=1,..,n-1$ and $s_1 \leq s_3$, then $w(K^n)-w(K^{n-1}) \geq 18+36(n-3)$. \end{lemma}
\begin{proof}

 The isotopy between $K^n$ and $K^{n-1}$ is depicted in Figure \ref{fig:Kn-1}. For the second part of the claim, note that $K^n$ and $K^{n-1}$ can each be decomposed into two tangles as in Figure 2 such that and one of these tangles, $T$, is the same for both. To obtain $K^n$, $T$ is inserted into $K^n/T$ at a level sphere $P$ with width $w(P)=2s_1+s_2+1$ and to obtain $K^{n-1}$, $T$ is inserted into $K^{n-1}/T$ at a level sphere $R$ with width $w(R)=s_2+1$. Note that if $X_{2,1}$, say, only has maxima, then the level sphere directly below it is isotopic to the level sphere that intersects the braid box in $a$ points. Therefore, it is easy to see that the computations below hold even if $X_{2,1}$ only contains maxima or $X_{2,2}$ only contains minima. We prove the case when $n\geq 4$. If $n=3$ the last 5 lines require minor modifications which we leave to the reader.

\vspace{-.5cm}

\begin{flalign*}
&w(K^n)-w(K^{n-1}) &\textrm{(by \ref{lem:addingtangle})}\\
=&w(K^n/T)+w(T)+(1+2s_1)(r-1)+(2s_1+s_2+1)\\&-w(K^{n-1} /T)-w(T)-1(r-1)-(s_2+1) &\textrm{(by \ref{lem:thickandthin})}\\
\end{flalign*}

\vspace{-1.25cm}

\begin{flalign*}
=&\frac{1}{2}[(a+s_1+1)^2+(b+s_1+1)^2-(2s_1+s_2+1)^2-(a+s_1-1)^2-(b+s_1-1)^2\\&+(s_2+1)^2]+2rs_1\\
\end{flalign*}

\vspace{-1.25cm}

\begin{flalign*}
 =&2(a+s_1)+2(b+s_1)-2s_1^2-2s_1s_2-2s_1+2rs_1&\textrm{(by \ref{rmk:easy}.1)}\\
\geq& 6s_1+4s_2-2s_1^2-2s_1s_2+2rs_1& \textrm{(as $s_i \geq 3$ for $i\leq n-1$)}\\
\geq& 30-2s_1^2-2s_1s_2+2rs_1&\textrm{(by \ref{rmk:easy}.2)}\\
\geq&30-2s_1^2-2s_1s_2+2s_1(s_2+2s_3+2s_4+...+2s_{n-1}+s_n)&\\
= &30-2s_1^2+2s_1(2s_3+2s_4+...+2s_{n-1}+s_n)&\textrm{($s_1 \leq s_3$)}\\
\geq&30-2s_1^2+2s_1(s_1+s_3+2s_4+...+2s_{n-1}+s_n)\\
=&30+2s_1(s_3+2s_4+...+2s_{n-1}+s_n)&\textrm{($s_i \geq 3$ for $i\leq n-1$)}\\
\geq&30+6(3+6(n-4)+1)\geq 18+36(n-3).
\end{flalign*}
\end{proof}

Now, we consider a second isotopy presented in Figure \ref{fig:Kn3}. We will need the following remark:

\begin{rmk}\label{rmk:iso} Raising a maxima above a minima  adds $4$ to the width. Raising a minima above a maxima  lowers the width by $4$. Passing two minima or two maxima past each other does not affect the width.
\end{rmk}

\begin{figure}
\begin{center} \includegraphics[scale=.4]{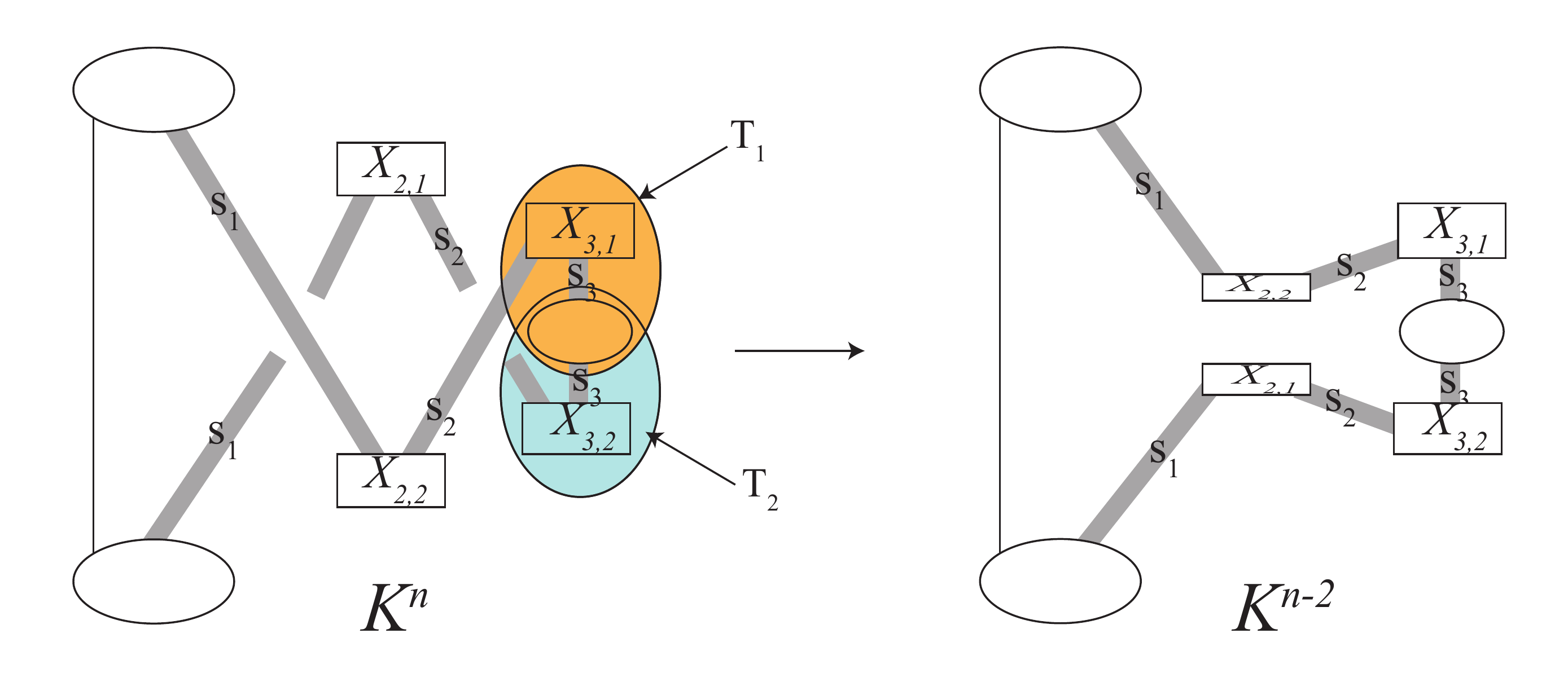}
\end{center}
\caption{} \label{fig:Kn3} \end{figure}

\begin{lemma}\label{lem:s1geqs3}
Let $K^n$ and $K^{n-2}$ be the projections of knots of type $n$ and $n-2$ respectively depicted in Figure \ref{fig:Kn3} where $n>2$ and $s_{i} \geq 3$ for $1 \leq i \leq n-1$. Then $K^n$ and $K^{n-2}$ are projections of the same knot. Furthermore, if $s_1\geq s_3$, then $w(K^n)-w(K^{n-2}) \geq max(72n-204,36)$.
\end{lemma}

\begin{proof}
Let $M_{i}$ (respectively $m_{i}$) denote the number of maxima (respectively minima) in braid box $X_{2,i}$ for $i=1,2$. Let $T_{1}$ and $T_{2}$ be the tangles illustrated in Figure \ref{fig:Kn3}. Let $M_{T_{i}}$ (respectively $m_{T_{i}}$) denote the number of maxima (respectively minima) in the tangle $T_{i}$ for $i =1,2$. We can think of the isotopy of $K^{n}$ to $K^{n-2}$ as done in two stages. First, vertically lower $X_{2,1}$ past all of $T_{1}$ until it lies strictly between $T_{1}$ and $X_{3,2}$. Second, vertically raise $X_{2,2}$ past all of $T_{2}$ and all of $X_{2,1}$ until it lies strictly between $T_{2}$ and $X_{3,1}$.

By Remark \ref{rmk:iso}, vertically lowering $X_{2,1}$ past all of $T_{1}$ changes the width by $4m_{1}M_{T_{1}} - 4M_{1}m_{T_{1}}$. Vertically raising $X_{2,2}$ past all of $T_{2}$ and all of $X_{2,1}$ changes the width by $4M_{2}m_{T_{2}} - 4m_{2}M_{T_{2}} + 4M_{2}m_{1} -4m_{2}M_{1}$. Thus,
$$w(K^{n})-w(K^{n-2})=-4m_{1}M_{T_{1}} + 4M_{1}m_{T_{1}} - 4M_{2}m_{T_{2}} + 4m_{2}M_{T_{2}} - 4M_{2}m_{1} +4m_{2}M_{1}.$$

Because $s_{1} + s_{2}$ strands are entering $X_{2,1}$ from below, $M_{1}= m_{1} + \frac{1}{2}(s_{1}+s_{2})$. Similarly $M_{2}= m_{2} - \frac{1}{2}(s_{1}+s_{2})$, $M_{T_{1}}= m_{T_{1}} + \frac{1}{2}(s_{2}+s_{3})$, and $M_{T_{2}}= m_{T_{2}} - \frac{1}{2}(s_{2}+s_{3})$. By substituting these values into the above equation and simplifying, we get:

\begin{flalign*}
&w(K^{n})-w(K^{n-2})&\\
&=2(m_{1}+m_{2})(s_{1}-s_{3}) + 2(s_{1}+s_{2})(m_{T_{1}}+m_{T_{2}})&\textrm{(as $s_1 \geq s_3$)}\\
&\geq 2(s_{1}+s_{2})(m_{T_{1}}+m_{T_{2}})&\\
\end{flalign*}

 If $n \geq 4$, then $m_{T_{1}} \geq \frac{1}{2}(s_{3}) + s_{4} + ... +s_{n-1} + \frac{1}{2}(s_{n})$ and $m_{T_{2}} \geq \frac{1}{2}(s_{2}) + s_{3} + ... +s_{n-1} + \frac{1}{2}(s_{n})$. Additionally, $s_{i} \geq 3$ for $1\leq i \leq n-1$ and $s_{n} \geq 1$. Thus, $m_{T_{1}} \geq 3(n-4)+2$ and $m_{T_{2}} \geq 3(n-3)+2$. Hence, $w(K^{n})-w(K^{n-2}) \geq 2(s_{1}+s_{2})(6n-17) \geq 12(6n-17) = 72n-204$.

If $n=3$ and $s_{3} \geq 3$, then $m_{T_{1}} \geq 0$ and $m_{T_{2}} \geq 3$. Hence, $w(K^{n})-w(K^{n-2}) \geq 2(s_{1}+s_{2})(3) \geq 12(3) = 36$.

If $n=3$ and $s_{3} = 1$, then $m_{T_{1}} \geq 0$ and $m_{T_{2}} \geq 2$. Since $s_{1} \geq 3$ and $s_{2} \geq 3$, then
$m_{1} \geq 0$ and $m_{2} \geq 3$. Hence, $w(K^{n})-w(K^{n-2}) \geq 2(0+3)(3-1) + 2(3+3)(0+2) = 36$.

Since $72n-204 \geq 36$ for $n\geq 4$, $w(K^{n})-w(K^{n-2}) \geq max(72n-204,36)$.

\end{proof}

\begin{thm}\label{thm:s1leqs3}
If $n > 2$ and $s_i\geq 3$ for $i=1,...,n-1$, then $w(\mathcal{K}^n)\leq w(K^n)-2n^2$.
\end{thm}

\begin{proof}
We will prove this result by induction. For $n=3$, apply Lemma \ref{lem:s1leqs3} if $s_1\leq s_3$ or Lemma \ref{lem:s1geqs3} if $s_1>s_3$. Then either $$w(\mathcal{K}^3)\leq w(K^3)-18\leq w(K^3)-2(3^2)\textrm{,     or}$$
$$w(\mathcal{K}^3)\leq w(K^3)-36 \leq w(K^3)-2(3^2).$$ Similarly if $n=4$, $$w(\mathcal{K}^4)\leq w(K^4)-18-36\leq w(K^3)-2(4^2)\textrm{,     or}$$
$$w(\mathcal{K}^4)\leq w(K^4)-72(4)+204=w(K^4)-84\leq w(K^3)-2(4^2).$$

Now assume that the width of a projection in the form $K^{n-1}$ can be decreased by $2(n-1)^2$ and the width of a projection in the form $K^{n-2}$ can be decreased by $2(n-2)^2$. Consider a knot of the form $K^n$ with $n \geq 5$.

Case 1: First suppose $s_1\leq s_3$. By Lemma \ref{lem:s1leqs3}, the projections $K^n$ and $K^{n-1}$ are isotopic. Therefore, $\mathcal{K}^n=\mathcal{K}^{n-1}$. Using the inequality in the lemma, it follows that
\begin{flalign*}
&w(\mathcal{K}^n)=w(\mathcal{K}^{n-1}) &\textrm{(by induction hyp.)}\\
&\leq w(K^{n-1})-2(n-1)^2 &\textrm{(by 3.2)}\\
&\leq w(K^{n}) -18 -36(n-3) -2(n-1)^{2} &\textrm{(as $n\geq 5$)}\\
&\leq w(K^{n}) -2n^{2}\\
\end{flalign*}

Case 2: Suppose $s_1\geq s_3$. Applying Lemma \ref{lem:s1geqs3}:

\begin{flalign*}
&w(\mathcal{K}^n)=w(\mathcal{K}^{n-2}) &\textrm{(by induction hyp.)}\\
&\leq w(K^{n-2})-2(n-2)^2 &\textrm{(by 3.4)}\\
&\leq w(K^{n}) -72n +204 -2(n-2)^{2} &\textrm{(as $n\geq 5$)}\\
&\leq w(K^{n}) -2n^{2}\\
\end{flalign*}

\end{proof}

\begin{thm}\label{thm:siis1}
If $n \geq 2$ and $s_i=1$ for some $i=1,...,n-1$, then $w(K^n)-w(\mathcal{K}^n)\geq 2n^2-2$.
\end{thm}

\begin{proof}
Suppose $s_p=1$ for some $1 \leq p \leq n-1$. In this case, $K^n$ is composite with summands of the form $K^p$ and $K^q$ (one or both summands might be the unknot) where $p+q=n$. Let $l=1+2s_1+...+2s_{p-1}$ and let $r$ be the sum of the number of critical points in braid boxes $X_{p+1,1},X_{p+1,1},..X_{n,1},X_{n,2}$.
Consider the tangle $T^q$ indicated in Figure \ref{fig:tangle} and the associated knot $K^q$. By Lemma \ref{lem:addingtangle}, it follows that $w(K^n)\geq w(K^p)+w(T^q)+l(r-1)+(l+1)$. On the other hand, $w(T^q)= w(K^q)+(r-1)$. We know that $w(\mathcal{K}^n)\leq w(\mathcal{K}^p)+w(\mathcal{K}^q)-2$ and, therefore,
$$w(K^n)-w(\mathcal{K}^n)\geq w(K^p)-w(\mathcal{K}^p)+w(K^q)-w(\mathcal{K}^q)+(l+1)r+2.$$ Note that $l\geq 2p-1$ and $r\geq 2q$.

\begin{figure}
\begin{center} \includegraphics[scale=.5]{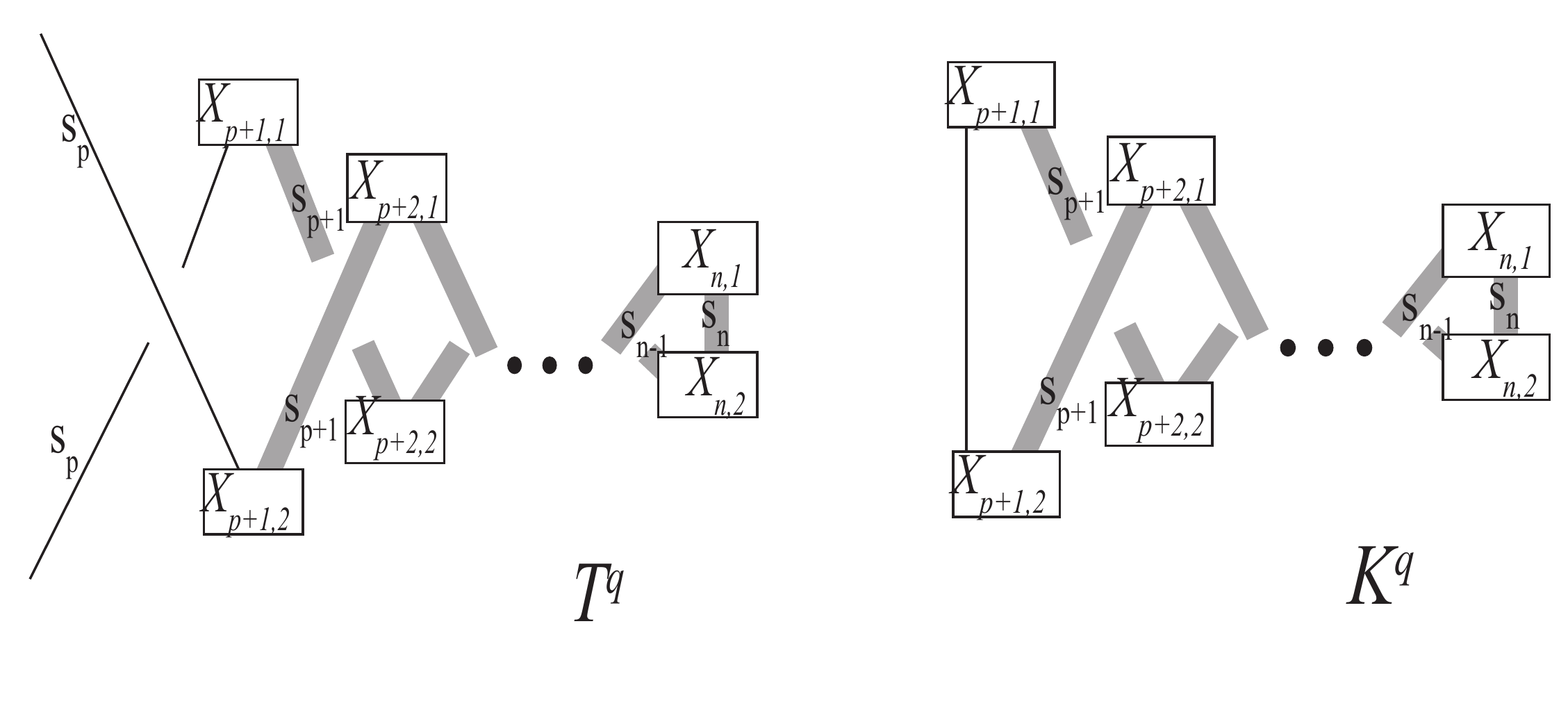}
\end{center}
\caption{} \label{fig:tangle} \end{figure}

Suppose first that $n=2$. In this case, $p=q=1$ and so $l=1$ and $r\geq 2$. Thus, $w(K^2)-w(\mathcal{K}^2)\geq (l+1)r+2\geq6= 2(2^2)-2$.

Suppose $n=3$, $p=1$ and $q=2$. If $r=4$, it is clear that $K^q$ is the unknot. So, by undoing two Reidemeister 1 moves, we obtain $w(K_3)-w(\mathcal{K}^3)\geq 16=2n^2-2$.
 If $r=6$, it is easy to check that $s_2=1$ as well. So, $K^q$ is composite. Therefore, by the previous case, $w(K^p)-w(\mathcal{K}^p)+w(K^q)-w(\mathcal{K}^q)+(l+1)r+2\geq 6+12+2\geq 16=2n^2-2$. If $r\geq 8$, it suffices to note that $(l+1)r+2\geq 18>2n^2-2$.

Suppose $n=3$, $p=2$ and $q=1$. If $s_{1}=1$, we can appeal to the previous case. Hence, we can assume $s_{1}\geq 3$ and therefore $l\geq7$. In this case, it is sufficient to note that $(l+1)r+2\geq 8(2)+2> 2n^2-2$.

We will prove the result by induction. Suppose that $n\geq 4$, the result holds for all $2\leq k\leq n-1$ and $K^n=K^p\#K^q$. In addition, if $s_i\geq 3$ for all $i=1,..,k-1$ and $k\geq 3$, then $w(K^k)-w(\mathcal{K}^k)\geq 2p^2-2$ by Theorem \ref{thm:s1leqs3}. If $p=1$ or $p-1$, using the induction hypothesis or Theorem \ref{thm:s1leqs3}, we obtain:
\begin{flalign*}
&w(K^p)-w(\mathcal{K}^p)+w(K^q)-w(\mathcal{K}^q)+(l+1)r+2\\
&\geq 2(n-1)^2-2+2(2(n-1))+2=2n^2-2\\
\end{flalign*}

Suppose $p=q=2$. Note that either we are in the previous case, or $l\geq 7$ so we obtain:
\begin{flalign*}
&w(K^p)-w(\mathcal{K}^p)+w(K^q)-w(\mathcal{K}^q)+(l+1)r+2\\
&\geq (8)(4)+2>2n^2-2\\
\end{flalign*}

Suppose $p=2$ or $p-1$ and $n> 4$. Either we can assume that $p=1$, a case we have already considered, or $s_1 \geq 3$ and thus $l+1 \geq 2p+4$. Using the induction hypothesis or Theorem \ref{thm:s1leqs3}, we obtain:
\begin{flalign*}
&w(K^p)-w(\mathcal{K}^p)+w(K^q)-w(\mathcal{K}^q)+(l+1)r+2\\
&\geq 2(n-2)^2+(2p+4)(2(n-p))=2[(n-2)^2+2(p(n-p))+4(n-p)]\\
&>2[(n-2)^2+2(2(n-2))+4]>2n^2-2\\
\end{flalign*}

Finally, if  $2 \leq p\leq n-2$ and $n>4$, we obtain:
\begin{flalign*}
&w(K^p)-w(\mathcal{K}^p)+w(K^q)-w(\mathcal{K}^q)+(l+1)r+2\\
&\geq 2p^2-2+2(n-p)^2-2+(2p)(2(n-p))+2=2n^2-2.\\
\end{flalign*}
\end{proof}

\begin{mr}
For all $n>2$, the examples of knot projections proposed in \cite{ST} which satisfy $w(K^n \#B^n)= w(K^n)$ also satisfy the inequality $w(K^n \#B^n)\geq w(\mathcal{K}^n)+ w(B^n)-2$ and are, therefore, not counterexamples to $w(\mathcal{K}\#\mathcal{K}')\leq w(\mathcal{K})+w(\mathcal{K}')-2$ as proposed.

\end{mr}

\begin{proof}
Note that all $s_i$ must be odd. If for some $1\leq i\leq n-1$ $s_i=1$, then, by Theorem \ref{thm:siis1}, $w(K^n)=w(K^n \#B^n) \geq w(\mathcal{K}^n)+2n^2-2 \geq w(\mathcal{K}^n)+w(B^n)-2$. Thus, we may assume $s_i\geq 3$ for all $i<n$.  In this case, $w(K^n)=w(K^n \#B^n) \geq w(\mathcal{K}^n)+2n^2 > w(\mathcal{K}^n)+w(B^n)-2$, by Theorem \ref{thm:s1leqs3}.

\end{proof}

\section{Extending Results and Open Questions}

In this section, we briefly consider a much larger class of knots containing the knots of type-$n$ described above. Suppose a knot $\mathcal{K}$ has an embedding $K$ as a wrapping number one companion of the unknot $U$, with $U$ in $n$-bridge position and some meridian disk of the unknotted torus intersects $K$ in a single point. See Figure \ref{fig:wrappingtorus}. We will call these knots generalized type-$n$ knots and the particular embedding will be called a generalized type $n$-embedding. Any knot for which the generalized type-$n$ embedding is in thin position would give an example where $w(\mathcal{K}\# \mathcal{K}')= max \{ \mathcal{K}, \mathcal{K}'\}$. In particular, if the companion unknot has bridge number $n$, then the connect sum of the knot with a knot for which $n$-bridge position and thin position coincide would give such an example. This leads us to the natural question:

\begin{qu}Is there a knot for which the generalized type-$n$ embedding is in thin position?\end{qu}

\begin{figure}
\begin{center} \includegraphics[scale=.3]{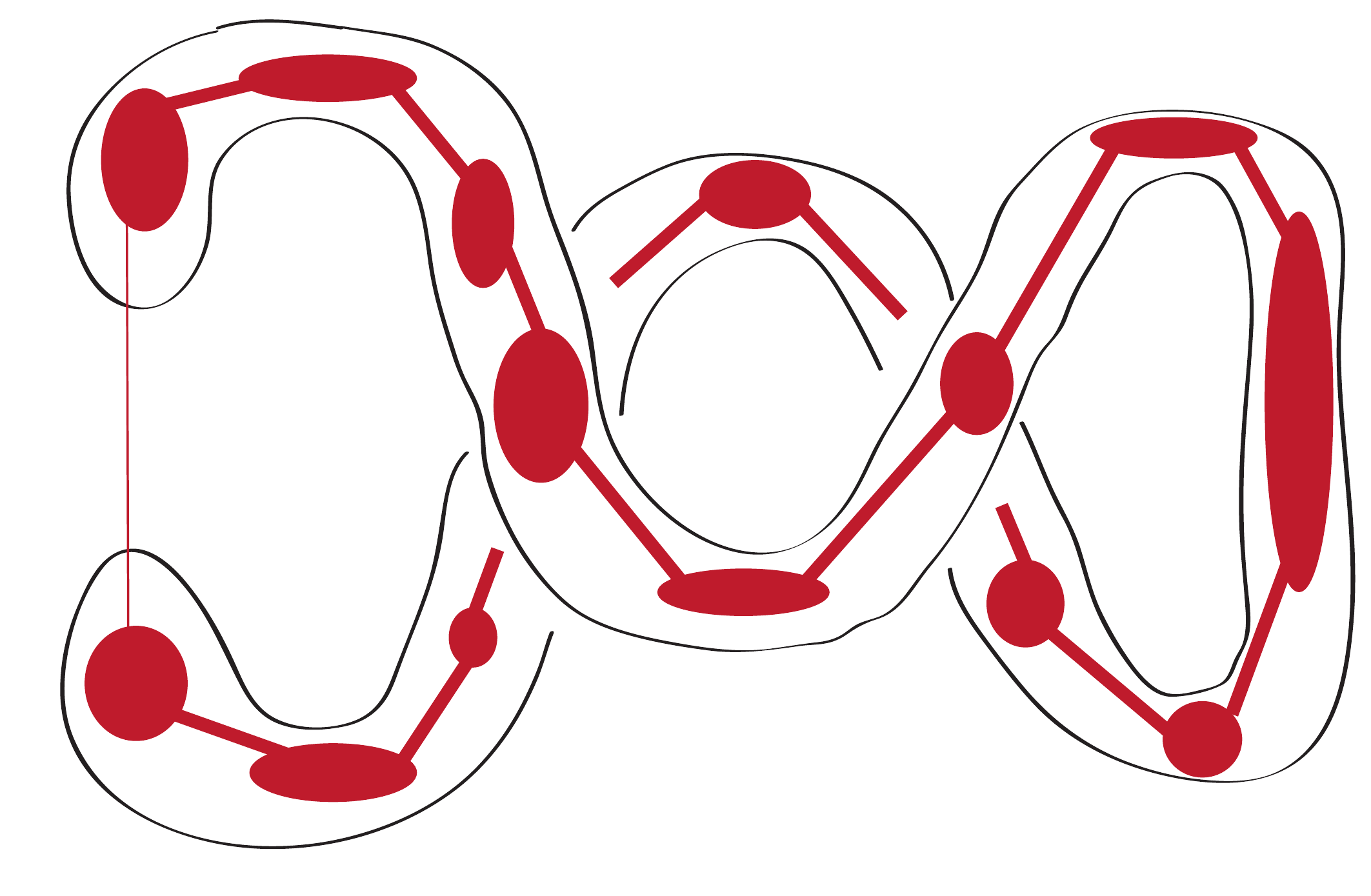}
\end{center}
\caption{The knot is depicted in red. Ovals represent tangles and thick lines represent multiple parallel strands.} \label{fig:wrappingtorus} \end{figure}

The previous section suggests that perhaps there are no such knots at least for $n>2$. Proving this more general result, though, seems very difficult. We can, however, generalize the results in the previous section to address some additional subsets of generalized type-$n$ knots and show that none of them gives a counterexample to width additivity. The proof uses techniques already introduced in the paper, so we will not provide details here.

First, consider generalized type-$n$ projections $K^n_g$, that satisfy all but the last requirement in the definition of a type-$n$ knot. See Figure \ref{fig:kng}.
\begin{figure}
\begin{center} \includegraphics[scale=.3]{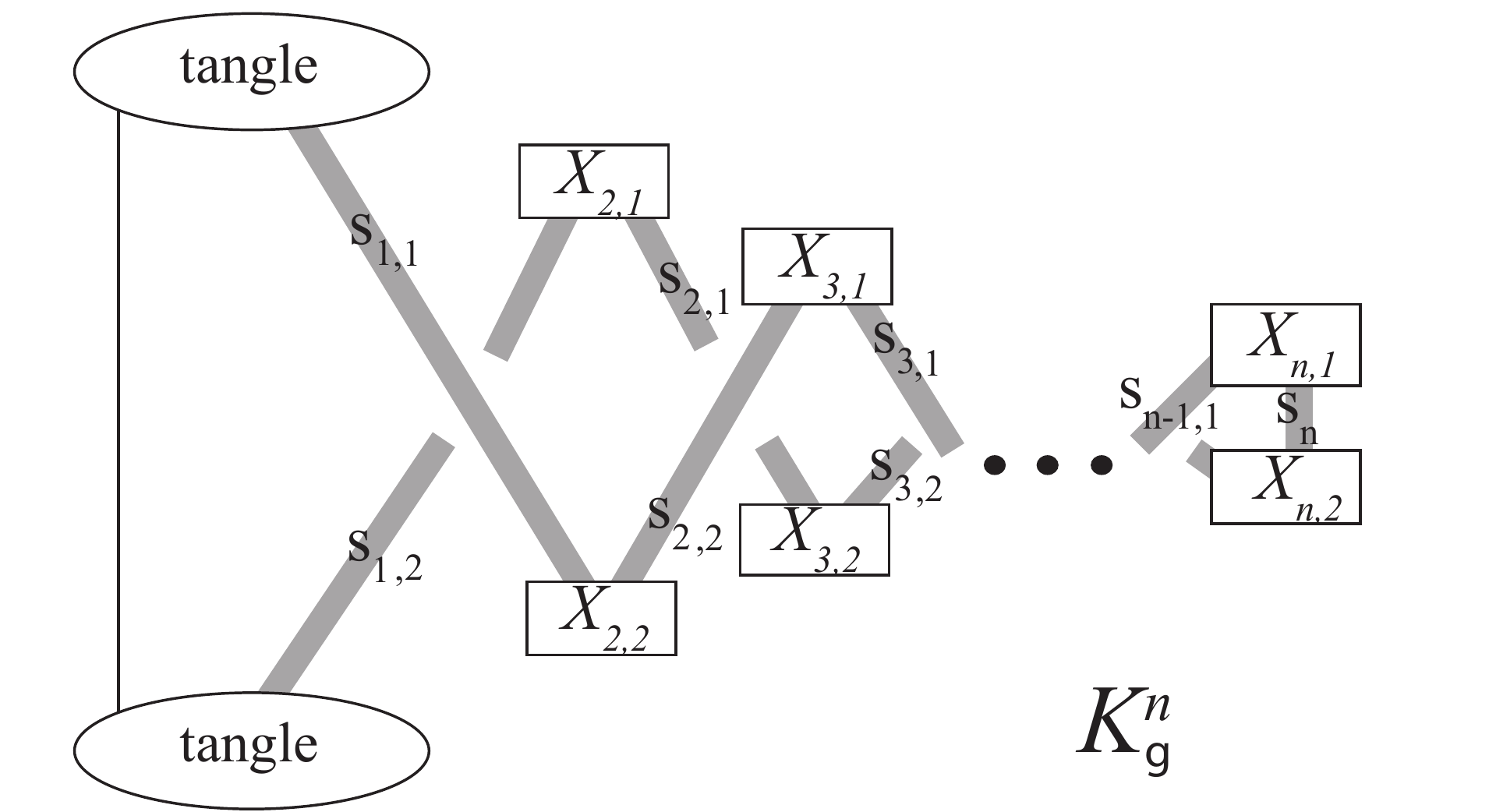}
\end{center}
\caption{} \label{fig:kng} \end{figure}

\begin{prop}
Let $K^n_g$ be the projection of the link depicted in Figure \ref{fig:kng}. If $n \geq 3$, and $s_{1,1}+s_{1,2} \leq s_{3,1}+s_{3,2}$, then $K^n_g$ is not in thin position. \end{prop}

\begin{proof} The proof is a natural generalization of Lemma \ref{lem:s1leqs3}.\end{proof}

\begin{prop}
Let $K^n_g$ be the projection of the link depicted in Figure \ref{fig:kng}. If $n \geq 3$, $s_{1,1}\geq s_{3,2}$ and $s_{1,2}\geq s_{3,1}$, then $K^n_g$ is not in thin position. \end{prop}

\begin{proof} The proof is a natural generalization of Lemma \ref{lem:s1geqs3}.\end{proof}

\begin{prop}
Let $K^n_g$ be the projection of the link depicted in Figure \ref{fig:kng}. If $n \geq 3$ and there exist $i,j$ such that $s_{i,j}<s_{k,l}$ for all $(k,l)\neq(i,j)$ and $i\geq3$, then $K^n_g$ is not in thin position. \end{prop}

\begin{proof} This follows from a modification of the proof of Lemma \ref{lem:s1geqs3}.\end{proof}

Unfortunately, the previous three propositions do not encapsulate all possibilities for $K^n_g$. This leads us to our next question.

\begin{qu}Let $K^n_g$ be the projection of the link depicted in Figure \ref{fig:kng}. If $n \geq 3$, can $K^n_g$ ever be in thin position?\end{qu}

If this question is answered in the negative, it would be most satisfying to also have a useful estimate on how much thinner we can make $K^n_g$.

We now further generalize the projections we consider. Let $K^n_l$ be a generalized type $n$-projection obtained from the definition of a type-$n$ projection by removing the last requirement (so we may have $s_{i,1} \neq s_{i,2}$) and replacing the first two requirements in the definition of a type-$n$ projection with the requirement that if $X_{i,1}$ is lower than and disjoint from braid box $X_{j,1}$, then braid box $X_{i,2}$ is higher than and disjoint from braid box $X_{j,2}$ and vice versa, see Figure \ref{fig:knlgen}.

\begin{figure}
\begin{center} \includegraphics[scale=.3]{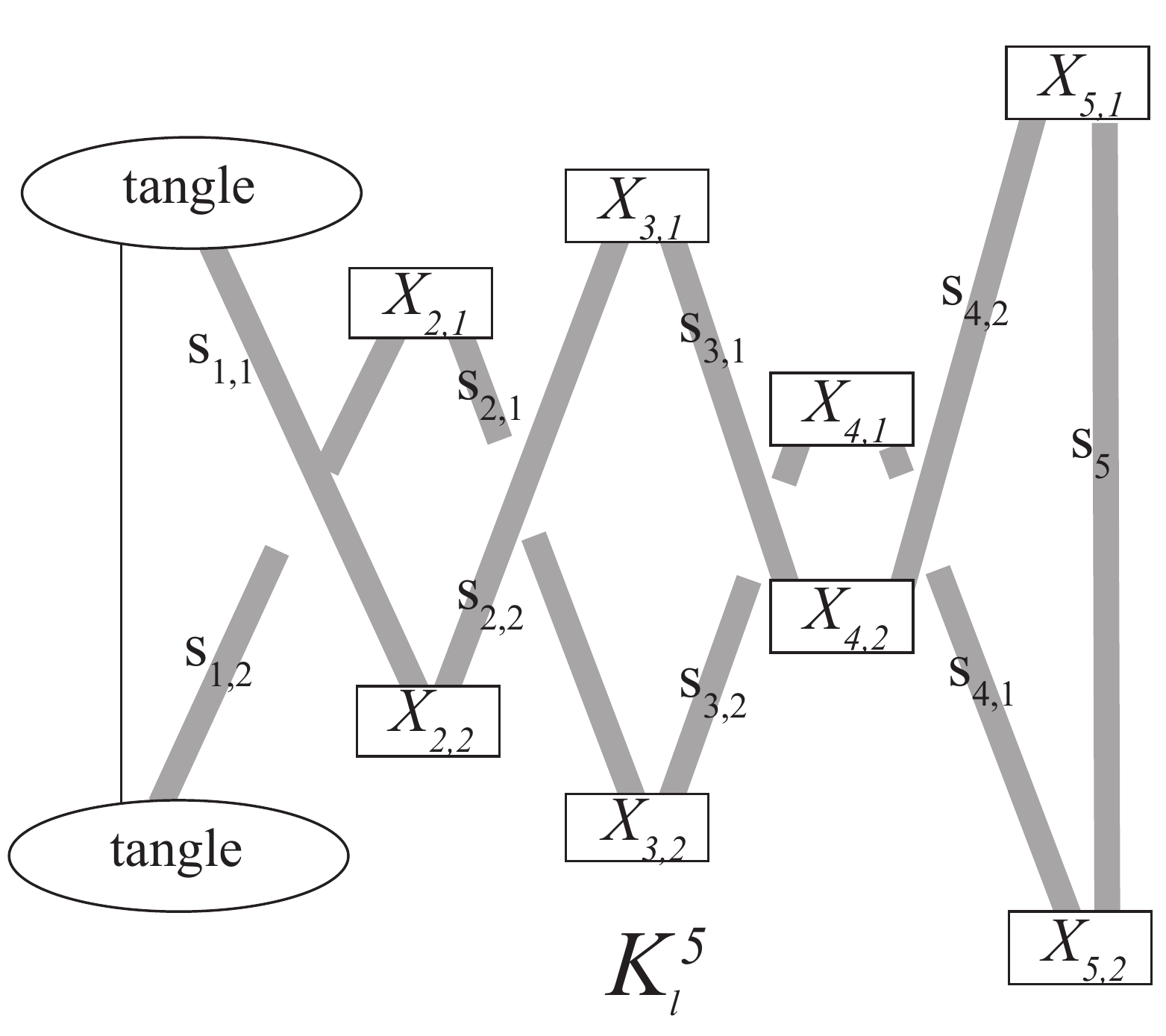}
\end{center}
\caption{} \label{fig:knlgen} \end{figure}

We can show that at least some of these projections cannot be in thin position. More precisely:
\begin{prop}
Let $K$ be the projection of the link depicted in Figure \ref{fig:knl}. In particular, assume that all tangles have heights disjoint from the heights of $X_{2,1}$ and $X_{2,2}$ and tangle $T$ has equal number $t$ of strands entering it from above and leaving it from below. Then $K$ is not in thin position. \end{prop}

\begin{proof} The proof is suggested by Figure \ref{fig:knl}. The computations are very similar to those in the proof of Lemma \ref{lem:s1geqs3} and show that the isotopy thins the projection by at least $8r+4$ where $r$ is the number of maxima in $T$. \end{proof}
\begin{figure}
\begin{center} \includegraphics[scale=.3]{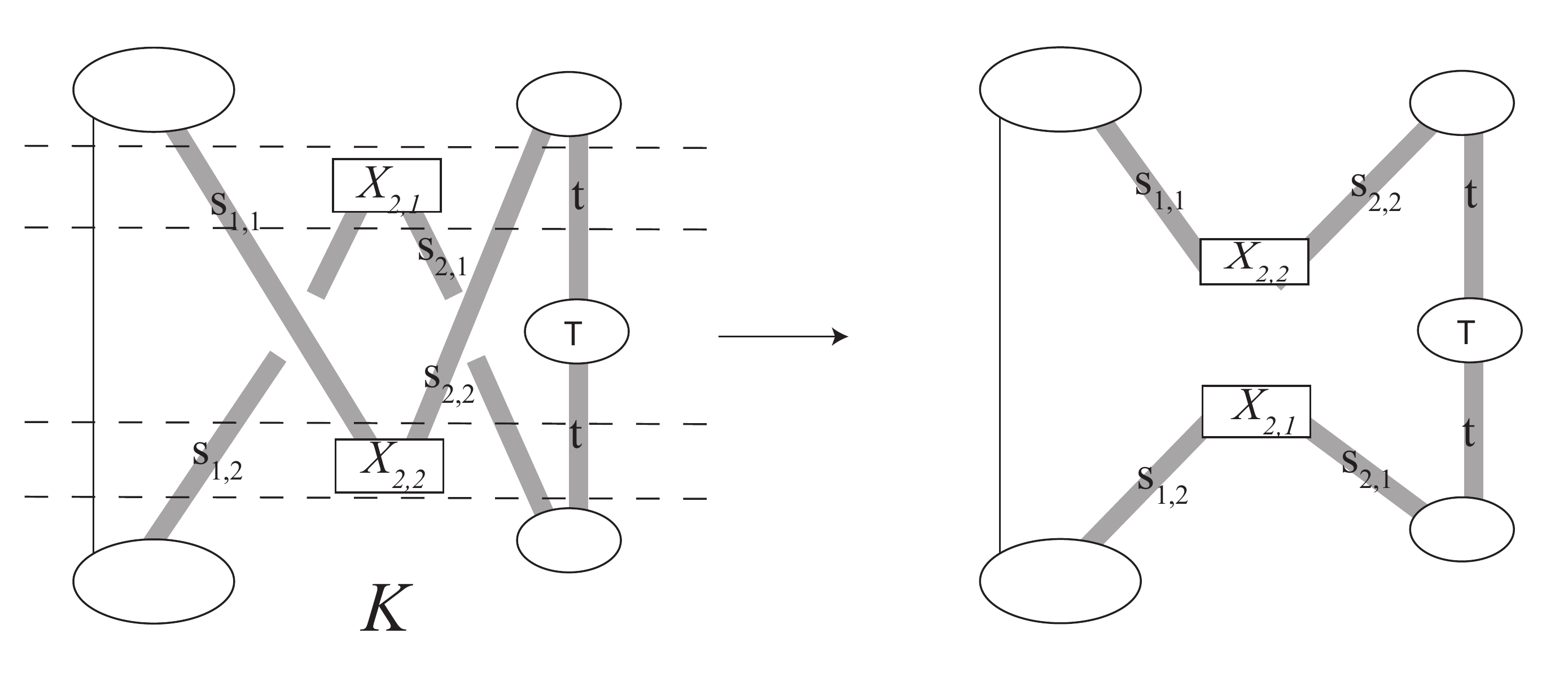}
\end{center}
\caption{} \label{fig:knl} \end{figure}

This result leads us to our last question.

\begin{qu}  Let $K^n_l$ be a projection defined above. Can $K^n_l$ be in thin position?\end{qu}

 \end{document}